\documentclass{amsart}
\usepackage{amssymb}
\usepackage{amsmath,amsfonts,amsthm,amstext}
\usepackage[all]{xy}
\newtheorem{theorem}{Theorem}[section]
\newtheorem{lemma}[theorem]{Lemma}
\newtheorem{prop}[theorem]{Proposition}
\newtheorem{cor}[theorem]{Corollary}

\numberwithin{equation}{section}

\newcommand{\epi}{\twoheadrightarrow}
\newcommand{\mono}{\rightarrowtail}

\newcommand{\R}{\mathbb{R}}
\newcommand{\Z}{\mathbb{Z}}

\theoremstyle{definition}

\begin{document}
\title{Homological finiteness properties of fibre products}
\author{ Dessislava H. Kochloukova, Francismar Ferreira Lima}
\address
{State University of Campinas (UNICAMP), SP, Brazil \\
email : desi@ime.unicamp.br
\newline
\\ Technical State University of Paran\'a (UTFPR), PR, Brazil, email : francismarferreiralima@gmail.com\\} 
\email{}
\subjclass[2010]{Primary 20J05; Secondary 20F05;}
\date{}
\keywords{}
\begin{abstract} We study the homological finiteness property $FP_n$ of fibre products. A homological version of the $n$-$(n+1)$-$(n+2)$ Conjecture is suggested and solved in some cases. Though the Homological 1-2-3 Conjecture is still open we prove a homological version of the Virtual Surjection Conjecture in the case of virtual surjection on pairs.  
\end{abstract}
\maketitle
\section{Introduction}

In this paper we study homological finiteness properties $FP_n$  of the fibre $P$ of two epimorphisms of groups $f_1 : G_1 \to Q$ and $f_2 :
G_2 \to Q$. In \cite{Benno} Kuckuck  studied the homotopical finiteness property $F_n$ of the fibre $P$.
The homotopical type $F_n$ was defined by Wall in \cite{Wall}. We recall that a group $G$ is of type $F_n$ if there is $K(G,1)$-complex with finite $n$-skeleton. For $n \geq 2$ a group $G$ has a homotopical type $F_n$ if and only if it is finitely presented and has homological type $FP_n$. The latest means there is a projective resolution of the trivial $\Z G$-module $\Z$ with finitely generated projectives in all dimensions $ \leq n$, for more details and properties on the homological property $FP_n$  we refer the reader to the Bieri book \cite{Bieribook}.

By definition, for  epimorphisms of groups $f_1 : G_1 \to Q$ and $f_2 :
G_2 \to Q$, the fibre product of $f_1$ and $f_2$ is
$$P = \{(g, h) | f_1(g) = f_2(h) \} \subseteq G_1 \times G_2.$$
Alternatively we say that $P$ is the  fibre product associated to the short exact sequences 
 $Ker(f_1) \hookrightarrow G_1 \stackrel{f_1}{\twoheadrightarrow} Q$  and  $Ker(f_2) \hookrightarrow G_2 \stackrel{f_2}{\twoheadrightarrow} Q$. 
In the case when both $G_1, G_2$ are finitely presented, $Q$ is of homotopical type $F_3$  and one of $Ker (f_1)$ and $Ker(f_2)$ is finitely generated Bridson, Howie, Miller and Short showed in \cite{B-H-M-S.1} that $P$ is finitely presented. This result is called the 1-2-3 Theorem or sometimes the Asymmetric  1-2-3 Theorem. A symmetric version when $f_1 = f_2$ was proved earlier by Baumslag, Bridson, Miller and Short in \cite{B-B-M-S}.  Some results on finite presentability of twisted fibre products were established by Mart\'inez-P\'erez in \cite{Conchita} and involved the use of the Bieri-Strebel-Neumann $\Sigma$-invariant.

 In \cite{Benno} Kuckuck suggested 
 
 \medskip
 {\bf The  $n$-$(n+1)$-$(n+2)$ Conjecture} 
 {\it  Let $N_1 \to G_1 \to Q$ and $N_2 \to G_2 \to Q$ be short exact sequences, where $Q$ of type $F_{n+2}$, $G_1$ and $G_2$ are groups of type $F_{n+1}$ and  $N_1$  is of type $F_{n}$. Then the fibre product 
$P$ is of type $F_{n+1}$.}
 
 \medskip
 In this paper we discuss a homological version of this conjecture.
 
\medskip
 {\bf The Homological $n$-$(n+1)$-$(n+2)$ Conjecture} {\it  Let $N_1 \to G_1 \to Q$ and $N_2 \to G_2 \to Q$ be short exact sequences, where $Q$ is of type $FP_{n+2}$, $G_1$ and $G_2$ are groups of type $FP_{n+1}$ and  $N_1$  is of type $FP_{n}$. Then the fibre product 
$P$ is of type $FP_{n+1}$.}

\medskip
One of the main results in \cite{Benno} is the technical \cite[Prop. 4.3]{Benno} and it is proved there  by purely topological methods using stacks of complexes and the Borel construction. Our first result, Theorem A, is a homological version of \cite[Prop. 4.3]{Benno}. We prove Theorem A by purely algebraic means (spectral sequences) and observe that the original proof in \cite{Benno} cannot be translated in homological language i.e. the fact that the groups are finitely presented was essentially used in \cite{Benno}. 

\medskip
{\bf Theorem A} {\it Let $n \geq 1$ be a natural number, $A \hookrightarrow B \twoheadrightarrow C$ a short exact sequence of groups with  $A$ of type $FP_n$ and $C$ of type $FP_{n+1}$. Assume there is another short exact sequence of groups $A \hookrightarrow B_0 \twoheadrightarrow C_0$ with $B_0$ of type $FP_{n+1}$ and that there is a group homomorphism $\theta: B_0 \rightarrow B$ such that $\theta|_A = id_A$,  i.e. there is a commutative diagram of homomorphisms of groups  $$\xymatrix{A \ \ar@{^{(}->}[r] \ar[d]_{id_A} & B_0 \ar@{->>}[r]^{\pi_0} \ar[d]_{\theta} & C_0 \ar@{.>}[d]^{\nu} \\ A \ \ar@{^{(}->}[r] & B \ar@{->>}[r]^{\pi} & C}$$  Then $B$ is of type $FP_{n+1}$.}

\medskip
The homotopical version  of Theorem A, \cite[Prop. 4.3]{Benno},  was used in  \cite{Benno}  to prove several results about the $n$-$(n+1)$-$(n+2)$ Conjecture. Here we adopt the same approach and following the recipe suggested in \cite{Benno} we deduce from Theorem A several results about the Homological $n$-$(n+1)$-$(n+2)$ Conjecture.

\medskip

{\bf Theorem B} {\it The Homological $n$-$(n+1)$-$(n+2)$ Theorem holds if the second sequence splits.}

\medskip
{\bf Theorem C} {\it If the  Homological $n$-$(n+1)$-$(n+2)$ Conjecture holds whenever $G_2$ is a finitely generated free group then it holds in general.}

\medskip
The proof of the following result  uses  some properties of the homological $\Sigma$-invariants defined by Bieri and Renz in \cite{B-R}. In Section \ref{prelim} we will revise the properties of the homological $\Sigma$-invariants that will be needed later.

\medskip
{\bf Theorem D}
{\it Let $n \geq 1$ be a natural number, $N_1 \hookrightarrow G_1 \stackrel{\pi_1}{\twoheadrightarrow} Q,$ $N_2 \hookrightarrow G_2 \stackrel{\pi_2}{\twoheadrightarrow} Q$ be short exact sequences of groups, where $G_1, G_2$ are of type $FP_{n+1}$, $Q$ is virtually abelian, $N_1$ is of type $FP_k$  and $N_2$ is of type $FP_l$ for some $k,l \geq 0$ with $k + l \geq n$. Then the fiber product $P$ of $\pi_1$  and $\pi_2$ is of type $FP_{n+1}$.}

\medskip Though the general case of the Homological 1-2-3 Conjecture is still open, we solve it in the case when $Q$ is finitely presented.

\medskip
{\bf Theorem E} {\it The Homological 1-2-3 Conjecture holds if $Q$ is finitely presented.}

\medskip
Our interest in the homological finiteness properties of fibre products stems from our interest 
in the homological finiteness  type of subgroups of direct products of groups. Some results about the homotopical type $F_n$ were conjectured  in the case of subdirect products of non-abelian limit groups by Dison in \cite[section~12.5]{Dison}  and  in the case of some special subdirect products of groups of type $FP_{\infty}$ by Kochloukova in \cite{Desi1}.
Limit groups were defined by Sela and studied  by Kharlampovich and Myasnikov under the name fully residually free groups.  The class of limit groups played an important role in the solution of the Tarski problem in \cite{K-M} and \cite{Sela}.
 The interest in the study of homological and homotopical properties of subdirect products derives from the fact that every finitely generated residually free group embeds  as a subgroup of a direct product of finitely many limit groups \cite{B-M-R}.

The homotopical type of subdirect products of groups  was conjectured in \cite{Benno}, where Kuckuck stated the following form of the Virtual Surjection Theorem.

\medskip
{\bf The Virtual Surjection Conjecture} {\it Let $n \geq 2$ be a natural number, $G_1, \ldots, G_k$ be groups of homotopical type $F_n$, where $n \leq k$ and $P \subseteq G_1 \times \ldots \times G_k$ be a subgroup that virtually surjects on every $n$ factors i.e. for every $1 \leq i_1 < \ldots < i_n \leq k$ the image of $P$ under 
  the canonical projection $P \to G_{i_1} \times G_{i_2} \times \ldots \times G_{i_n}$ has finite index. Then $P$ is of type $F_n$.}

\medskip
In   \cite{B-H-M-S.1}  Bridson, Howie, Miller and Short showed that the Virtual Surjection Conjecture holds for $n = 2$ and this was deduced as a corollary of the 1-2-3 Theorem.
 This was later generalised in \cite{Benno}, where Kuckuck proved that if the  $(n-1)$-$n$-$(n+1)$ Conjecture holds when $Q$ is virtually nilpotent then the Virtual Surjection Conjecture holds in general.  In \cite{B-H-M-S.2} Bridson, Howie, Miller and Short proved that if $P$ is a finitely presented subdirect product of non-abelian limit groups $G_1, \ldots, G_k$ such that  $P \cap G_i \not= \emptyset$ for every $1 \leq i \leq k$ then $P$ virtually surjects on pairs. Later in \cite{Desi1}, Kochloukova showed   that if furthermore $P$ is of type $FP_n$ for some $n \leq k$ then  $P$  virtually surjects on every $n$ factors.
In this paper we suggest the following homological version of the Virtual Surjection Conjecture.

\medskip

{\bf The Homological Virtual Surjection Conjecture} {\it Let $n \geq 2$ be a natural number and $G_1, \ldots, G_k$ be groups of homological type $FP_n$, where $n \leq k$  and $P \subseteq G_1 \times \ldots \times G_k$ be a subgroup that virtually surjects on every $n$ factors. Then $P$ is of type $FP_n$.}

\medskip
The first part of Theorem F is  a homological version of \cite[Thm.~3.10]{Benno}. The second part of Theorem F follows from the first part, Theorem E and the fact that every virtually nilpotent group is finitely presented.

\medskip
{\bf Theorem F} {\it If the Homological $(n-1)$-$n$-$(n+1)$ Conjecture holds for $Q$ virtually nilpotent then the Homological Virtual Surjection Conjecture holds in general. In particular, the Homological Virtual Surjection Conjecture holds for $n =2$ i.e.  for groups that virtually surject on pairs.}

\medskip

Finally we note that some results on homological finiteness properties of fibre sums of Lie algebras and subdirect sums of Lie algebras  were recently established by Kochloukova and Mart\'inez-P\'erez in \cite{K-MP}. Though in the Lie algebra case there are no homotopic methods available, a version of the 1-2-3 Theorem for Lie algebras was proved in \cite{K-MP}.

\section{Preliminaries on the homological type $FP_m$ and homological $\Sigma$-invariants} \label{prelim}

 \subsection{Preliminaries on the homological type $FP_m$} 

If not otherwise stated the modules considered in this paper are left ones. 

\medskip
{\bf Definition.} {\it Let $S$ be an associative ring with 1. An $S$-module $M$ is said to be of type $FP_n$ if there is a projective resolution
$$
 \ldots \to P_i \to P_{i-1} \to \ldots \to P_0 \to M \to 0
 $$
 with $P_i$ finitely generated for all $i \leq n$.
 We say that a group $G$ is of type $FP_n$ if the trivial $\Z G$-module $\Z$ is of type $FP_n$.}

\medskip



We will need later the following criterion for modules of type $FP_n$.

\begin{lemma} \label{biericriterion} \cite[Prop.~1.2, Thm~1.3+remarks]{Bieribook} Let $S$ be an associative ring with 1 and $n \geq 1$ be a natural number. The following are equivalent for  $S$-module $M$ :

1) $M$ is of type $FP_n$;

2) for a direct product $\prod S$ of arbitrary many copies of $S$ we have $Tor_k^S(\prod S, M)  = 0$ for $1 \leq k \leq n-1$ and $M$ is finitely presented as   $S$-module;

3) the functor $Tor_k^S( - , M)$ commutes with arbitrary direct product for $0 \leq k \leq n-1$.
\end{lemma}

{\bf Remark} We will apply the above lemma for $S = \Z G$, where $G$ is a finitely generated group and for $M = \mathbb{Z}$ the trivial $\Z G$-module. In this case $M$ is automatically finitely presented as $S$-module.
\medskip

The following result is well known and can be deduced after making appropriate modifications to the proof of \cite[Prop.~2.7]{Bieribook}, which uses spectral sequences and Lemma \ref{biericriterion}. A detailed proof can be found in Lima's PhD thesis \cite{Fr}.

\begin{prop} \label{shortexact} Let $A \to B \to C$ be a short exact sequence of groups.

a) if both $A$ and $C$ are of type $FP_n$ then $B$ is of type $FP_n$;

b) if $A$ is of type $FP_n$ and $B$ is of type $FP_{n+1}$ then $C$ is of type $FP_{n+1}$.
\end{prop}

\subsection{Preliminaries on the homological $\Sigma$-invariants}  For a finitely generated group $G$ we define $S(G) = Hom (G, \R) \setminus \{ 0 \} / \sim$, where for two characters $\chi_1$, $\chi_2 \in Hom (G, \R) \setminus \{ 0 \}$ we have $\chi_1 \sim \chi_2$ if there is a positive real number $r$ such that $\chi_1 = r \chi_2$. We write $[\chi]$ for the equivalence class of $\chi$ with respect to $\sim$. Thus 
$$S(G) \simeq S^{n-1},$$
where $n$ is the torsion-free rank of the abelianization of $G$. 
The n-dimensional Bieri-Renz $\Sigma$-invariant is defined by 
$$
\Sigma^n(G, \Z) = \{ [\chi] \mid \Z \hbox{ is of type } FP_{n} \hbox{ as } \Z G_{\chi}\hbox{-module} \},
$$
where $G_{\chi}$ is the monoid  $ \{ g \in G \mid \chi(g) \geq 0 \}$.
The following results  will be used later in the paper.


\begin{theorem}\cite[Theorem 5.1]{B-R} \label{teo-subesf.grande.contida.em.sigma}
Let $n \geq 1$ be a natural number, $G$ a group of type $FP_n$ and $N \subseteq G$  a normal subgroup such that  $G/N$ is abelian.  Then, $N$ is of type $FP_n$ if and only if $$S(G,N) \subseteq \ \Sigma^n(G, \Z).$$
\end{theorem}

The following result was published in \cite{G}. As stated in \cite{G} the result was proved  (unpublished) by Meinert and generalizes ideas from \cite{Meinert}. 

\begin{theorem} \cite[Lemma~9.1]{G} \label{teo-desig.meinert}
Let $G_1, G_2$ be groups of type  $FP_n$ with $n \geq 1$  and $\chi: G_1 \times G_2 \rightarrow \R$ be a character such that $\chi|_{(G_1 \times \textbf{1})} \neq 0 \ \hbox{ and } \ \chi|_{(\textbf{1} \times G_2)} \neq 0$. If $[\chi|_{(G_1 \times \textbf{1})}] \in \Sigma^k(G_1, \Z)$ and $[\chi|_{(\textbf{1} \times G_2)}] \in \Sigma^l(G_2, \Z)$ for some  $k, l \geq 0$ with $k+l < n$, then $[\chi] \in \Sigma^{k+l+1}(G_1 \times G_2, \Z)$.
\end{theorem}

The above result was one of the reasons to believe in special formula for calculating $\Sigma$-invariants of direct product of groups, known as the   direct product conjecture for sigma invariants. As shown  by Bieri and Geoghegan in \cite{B-G}  this conjecture holds for the invariants $\Sigma^n(G, R)$, where $R$ is a field and $\Sigma^n(G, R)  = \{ [\chi] \mid R \hbox{ is of type } FP_{n} \hbox{ as } R G_{\chi}\hbox{-module} \}$. It turned out the conjecture is wrong in general, Sch\"utz proved in \cite{S} that the conjecture does not hold for $\Sigma^n(G, \Z)$ when $n \geq 4$.

\section{Homological version of a result of Kuckuck}

In this section we prove a homological version of  \cite[Prop. 4.3]{Benno}  with algebraic methods.  The starting point is the following result that is based on the classical Lyndon-Hoschild-Serre spectral sequence.

\begin{theorem} \label{teo-1.general}
Let $I$ be an index set, $n \geq 1$ a natural number  and $A \hookrightarrow B \twoheadrightarrow C$ a short exact sequence of groups. Furthermore we assume that $A$  is of type $FP_n$  and $B$  is of type $FP_{n+1}$, $M$ is a free $\Z B$-module and consider the LHS spectral sequence $$E^2_{p,q} = H_p(C, H_q(A, \displaystyle\prod_{\alpha \in I}M_{\alpha}))$$
converging to $H_{p+q}(B, \displaystyle\prod_{\alpha \in I}M_{\alpha})$, where $M_{\alpha} = M$ for $\alpha \in I$. Then $$E_{n+1,0}^{n+1} = E_{n+1, 0}^2 = H_{n+1}(C, H_0(A, \displaystyle\prod_{\alpha \in I}M_{\alpha})), E_{0,n}^{n+1} = E_{0,n}^2 = H_0(C, H_n(A, \displaystyle\prod_{\alpha \in I}M_{\alpha}))$$ and the differential  $$d_{n+1,0}^{n+1}: H_{n+1}(C, H_0(A, \displaystyle\prod_{\alpha \in I}M_{\alpha})) \longrightarrow H_0(C, H_n(A, \displaystyle\prod_{\alpha \in I}M_{\alpha}))$$ is surjective.
\end{theorem}

\begin{proof} 
First we assume  $n = 1$. In this case  $d_{2,0}^2 : E_{2,0}^2 \to E_{0,1}^2$. Since $B$ is $FP_2$ by  Lemma \ref{biericriterion} we have
 $H_1(B, \displaystyle\prod_{\alpha \in I}M_{\alpha}) = \displaystyle\prod_{\alpha \in I} H_1(B, M_{\alpha}) = 0$, hence by the convergence of the LHS spectral sequence  $E_{0,1}^{\infty} = 0$. Note that all differentials that enter and leave $E_{0,1}^j$ are zero if $j \geq 3$, so $0 =E_{0,1}^{\infty} = E_{0,1}^{3} = E_{0,1}^2/ Im (d_{2,0}^2)$, hence $d_{2,0}^2 $ is surjective.

From now on we assume that $n \geq 2$. We split the proof in several steps. 

{\bf Step 1.} Note that by Proposition \ref{shortexact} b) $C$ is of type $FP_{n+1}$.
Since $A$  is of type $FP_n$, by Lemma \ref{biericriterion}, we have that $$H_q(A, \displaystyle\prod_{\alpha \in I}M_{\alpha}) = \displaystyle\prod_{\alpha \in I}H_q(A, M_{\alpha}) \textrm{ for } 0 \leq q \leq n-1.$$ Now since  
 $M_{\alpha} = M$  is a free  $\Z B$-module for $\alpha \in I$, we have $H_q(A, M_{\alpha}) = \textbf{0}$ for $q \geq 1$ and  $$M \cong \displaystyle\bigoplus_{\beta \in J}(\Z B)_{\beta} \hbox{ for an index set }J,$$where $(\Z B)_{\beta} = \Z B$ for $ \beta \in J$. Since direct sum commutes with tensor product  for all  $ \alpha \in I$, we have  $$H_0(A,M_{\alpha}) \cong \Z \otimes_{\Z A}(\displaystyle\bigoplus_{\beta \in J}(\Z B)_{\beta}) \cong \displaystyle\bigoplus_{\beta \in J}(\Z (B/A))_{\beta} \cong \displaystyle\bigoplus_{\beta \in J}(\Z C)_{\beta} =: \tilde{M},$$
 where $(\Z C)_{\beta} = \Z C$ for  $\beta \in J$. Thus 
\begin{equation} \label{eq-E2.general}
E_{p,q}^2 = \left\{
\begin{array}{ccccccc}
\textbf{0}& \mbox{ if } & 1 \leq q \leq n-1\\
H_p(C, \displaystyle\prod_{\alpha \in I} \tilde{M}_{\alpha}) & \mbox{ if } & q = 0 
\end{array}
\right.
\end{equation} 
where $\tilde{M}_{\alpha} = \tilde{M}$ for all $ \alpha \in I$ and $\tilde{M}$ is a free $\Z C$-module.
Since $C$ is of type $FP_{n+1}$, by Lemma \ref{biericriterion} 
\begin{equation} \label{eq-homol.comut.com.prod.general}
H_p(C, \displaystyle\prod_{\alpha \in I}\tilde{M}_{\alpha}) \cong \displaystyle\prod_{\alpha \in I} H_p(C, \tilde{M}_{\alpha}) \textrm{ for } 0 \leq p \leq n
\end{equation} 
 and since $\tilde{M}_{\alpha}$ is a free $\Z C$-module,  we have that
\begin{equation} \label{kumon34} H_p(C, \tilde{M}_{\alpha}) = \textbf{0}, \hbox{ for all } \alpha \in I  \textrm{ and } p \geq 1.\end{equation}  
It follows by (\ref{eq-homol.comut.com.prod.general}) and (\ref{kumon34}) that
\begin{equation} \label{eq-homol.com.M.é.0.general}
H_p(C, \displaystyle\prod_{\alpha \in I}\tilde{M}_{\alpha}) = \textbf{0}, \textrm{ for } 1 \leq p \leq n.
\end{equation} 
Thus we obtain by (\ref{eq-E2.general}), that 
\begin{equation} \label{eq-E2.general.segundo}
E_{p,q}^2 = \textbf{0}, \mbox{ if } 1 \leq q \leq n-1, \textrm{ or } q = 0 \ \hbox{ and } \ 1 \leq p \leq n.
\end{equation} 

{\bf Step 2.}
Consider the differentials $$\xymatrix{E^i_{i,n+1-i} \ar[rr]^{d^i_{i,n+1-i}} & & E^{i}_{0,n} \ar[r]^-{d^i_{0,n}} & E^i_{-i,n+i-1} = \textbf{0}},$$ 
where $ E^i_{-i,n+i-1} = 0 $ since $ - i < 0$.
By definition
\begin{equation} \label{eq-Ei+1.=.Ei.general}
E^{i+1}_{0,n} = \frac{ker(d^i_{0,n})}{im(d^i_{i,n+1-i})} = \frac{E^i_{0,n}}{im(d^i_{i,n+1-i})}.
\end{equation} 
By (\ref{eq-E2.general.segundo})
\begin{equation} \label{eq-E2.i.general}
E_{i,n+1-i}^2 = \textbf{0}, \mbox{ if } 2 \leq i \leq n.
\end{equation}
On other hand for $i \geq n+2$ we have $n+1-i < 0$ and this implies that $E^2_{i,n+1-i} = \textbf{0}$ for $i \geq n+2$. Using this and  (\ref{eq-E2.i.general}),  we conclude that $E^2_{i,n+1-i} = {0} $ if $ i \neq n+1$ and $ \ i \geq 2.$
Hence $$ E^i_{i,n+1-i} = \textbf{0} \textrm{ if } i \neq n+1 \ \hbox{ and } \ i \geq 2, $$ and this implies that \begin{equation} \label{kumon201} im(d^i_{i,n+1-i}) = \textbf{0} \textrm{ if } i \neq n+1 \ \hbox{ and } \ i \geq 2.\end{equation} By (\ref{eq-Ei+1.=.Ei.general}) and (\ref{kumon201}) we obtain that $$E^{i+1}_{0,n} = E^i_{0,n}, \textrm{ if } i \neq n+1 \ \hbox{ and } \ i \geq 2,$$ hence
\begin{equation} \label{eq-igual.E0,s.general}
E^2_{0,n} = E^3_{0,n} = \ldots = E^n_{0,n} = E^{n+1}_{0,n} \ \ \ \ \hbox{ and } \ \ \ \ E^{n+2}_{0,n} = E^{n+3}_{0,n} = \ldots = E^{\infty}_{0,n}.
\end{equation}
This implies that  $d^{n+1}_{n+1,0} : E_{n+1,0}^{n+1} \to E^{n+1}_{0,n}$  has as codomain $E_{0,n}^{2} = H_0(C, H_n(A, \displaystyle\prod_{\alpha \in I}M_{\alpha}))$.

{\bf Step 3.} 
We will show that $E^{n+1}_{n+1,0} = E^2_{n+1,0}$.

Consider the differentials $$\xymatrix{E^i_{n+1+i,1-i} \ar[rr]^{d^i_{n+1+i,1-i}} & & E^i_{n+1,0} \ar[rr]^-{d^i_{n+1,0}} & & E^i_{n+1-i,i-1}}$$ 
For $i \geq 2$,  we have $E^i_{n+1+i,1-i} = \textbf{0}$, since $1-i < 0$. Then, $im(d^i_{n+1+i,1-i}) = \textbf{0}$ and so 
\begin{equation} 
E^{i+1}_{n+1,0}  = \frac{ker(d^i_{n+1,0})}{im(d^i_{n+1+i,1-i})} = ker(d^i_{n+1,0}).
\end{equation}
Now, by (\ref{eq-E2.general.segundo}), $E^i_{n+1-i,i-1} = \textbf{0}$ if $1 \leq i-1 \leq n-1$. Then $ker(d^i_{n+1,0}) = E^i_{n+1,0}$ if $2 \leq i \leq n$ and this implies $E^{i+1}_{n+1,0} = E^i_{n+1,0}$ if $2 \leq i \leq n$. Hence $$E^2_{n+1,0} = E^3_{n+1,0} = \ldots = E^n_{n+1,0} = E^{n+1}_{n+1,0}$$ and so $$E^{n+1}_{n+1,0} = E^2_{n+1,0} = H_{n+1}(C, H_0(A, \displaystyle\prod_{\alpha \in I}M_{\alpha})).$$

{\bf Step 4.} By the convergence of the LHS spectral sequence there is a filtration
\begin{equation} \label{eq-filtr.general}
\textbf{0} = \Phi^{-1}H_n \subseteq \Phi^{0}H_n \subseteq \ldots \subseteq \Phi^{n-1}H_n \subseteq \Phi^{n}H_n = H_n
\end{equation}
such that
\begin{equation} \label{eq-E.infty.é.quoc.da.filtr.general}
E^{\infty}_{p,q} \cong \Phi^pH_n/\Phi^{p-1}H_n \textrm{ for }   p+q=n,
\end{equation}
 where $H_n$ denotes $H_n(B, \displaystyle\prod_{\alpha \in I}M_{\alpha})$.  
By (\ref{eq-E2.general.segundo})   we deduce that $E_{p,q}^2 = \textbf{0}$ if $p + q = n \ \hbox{ and } \ p \neq 0$. Hence 
\begin{equation} \label{eq-E.infty.é.0.general}
E^{\infty}_{p,q} = \textbf{0}, \textrm{ if } p+q = n \ \hbox{ and } \ p \neq 0.
\end{equation}
Using the filtration (\ref{eq-filtr.general}), by (\ref{eq-E.infty.é.quoc.da.filtr.general})  and (\ref{eq-E.infty.é.0.general}),  we get  $$\textbf{0} = \Phi^{-1}H_n \subseteq \Phi^0H_n = \ldots = \Phi^{n-1}H_n = \Phi^{n}H_n = H_n.$$ Then \begin{equation} \label{kumon1} E^{\infty}_{0,n} \cong \Phi^0H_n/\Phi^{-1}H_n = H_n = H_n(B, \displaystyle\prod_{\alpha \in I}M_{\alpha}).\end{equation}
Note that by now we have used only the fact that $A$ is $FP_n$ and $C$ is $FP_{n+1}$ but we have {\bf not used} that  $B$ is of type $FP_{n+1}$. Since $B$ is of type $FP_{n+1}$,
by Lemma  \ref{biericriterion} and since for all $ \alpha \in I, M_{\alpha} = M$ is a free $\Z B$-module we obtain that $$H_n(B, \displaystyle\prod_{\alpha \in I}M_{\alpha}) \cong \displaystyle\prod_{\alpha \in I}H_n(B, M_{\alpha}) = \displaystyle\prod_{\alpha \in I} 0 = \textbf{0} $$ and so 
\begin{equation} \label{eq-Einfty.é.0.general.segundo}
E^{\infty}_{0,n} = \textbf{0}
\end{equation}
By   (\ref{eq-igual.E0,s.general}) and (\ref{eq-Einfty.é.0.general.segundo}) we have
\begin{equation} \label{eq-Es+2.general}
E^{n+2}_{0,n} = \textbf{0}.
\end{equation}
Consider  the differentials $$\xymatrix{E^{n+1}_{n+1,0} \ar[rr]^{d^{n+1}_{n+1,0}} & & E^{n+1}_{0,n} \ar[rr]^-{d^{n+1}_{0,n}} & & E^{n+1}_{-n-1,2n} = \textbf{0}}$$ and note that $$E^{n+2}_{0,n} = \frac{ker(d^{n+1}_{0,n})}{im(d^{n+1}_{n+1,0})} = \frac{E^{n+1}_{0,n}}{im(d^{n+1}_{n+1,0})}.$$ Then by (\ref{eq-Es+2.general}) we obtain that  $d^{n+1}_{n+1,0}$ is surjective.

\end{proof}

\begin{prop} \label{homology1}
Let  $n \geq 1$ be a natural number, $B$ a group with a normal subgroup $A$ such that $A$ is of type $FP_n$ and $C = B/A$ is of type $FP_{n+1}$.  Then
$B$ is of type $FP_{n+1}$ if and only if for any  direct product the map
$$
d_{n+1,0}^{n+1}: H_{n+1}(C, H_0(A, \displaystyle\prod \Z B)) \longrightarrow H_0(C, H_n(A, \displaystyle\prod \Z B))
$$
is surjective , where $d_{n+1,0}^{n+1}$ is the differential from LHS spectral sequence 
$
E_{p,q}^2 = H_p(C, H_q(A, \prod \mathbb{Z} B))$ converging to $H_{p+q}(B, \prod \mathbb{Z} B)$.
\end{prop}

\begin{proof} 
Note that one of the direction is precisely Theorem \ref{teo-1.general} . Assume from now on that the differential  $d_{
n+1,0}^{n+1}$ is surjective.
We apply the  notations of the proof of Theorem \ref{teo-1.general}   with $M =  \mathbb{Z} B$. Note that in the proof of (\ref{kumon1}) we have not used that $B$ is of type $FP_{n+1}$. Then we can apply (\ref{kumon1}) for $M = \Z B$ and deduce by Lemma \ref{biericriterion} that $$B \hbox{  is of type }FP_{n+1} \hbox{ if and only if }E_{0,n}^{\infty} = 0.$$  Note that in (\ref{eq-igual.E0,s.general}) of the proof of Theorem \ref{teo-1.general}  we have not used that $B$ is $FP_{n+1}$. By (\ref{eq-igual.E0,s.general}) we have 
$$
E^{2}_{0,n} = E_{0,n}^{n+1}, E_{0,n}^{n+2} = E^{\infty}_{0,n}.
$$ 
Thus $$B \hbox{ is of type }FP_{n+1} \hbox{ if and only if }0 = E_{0,n}^{n+2}.$$ 
Since $d^{n+1}_{0,n} : E_{0,n}^{n+1} \to E^{n+1}_{-n-1, 2n} = 0$ is the zero map, $E_{0,n}^{n+2} = ker(d_{0,n}^{n+1})/ im (d_{n+1, 0}^{n+1}) = E_{0,n}^{n+1} / im (d_{n+1, 0}^{n+1})$ is the cokernel of $d_{n+1,0}^{n+1}$. Thus $$0 = E_{0,n}^{n+2} \hbox{ if and only if }d_{n+1,0}^{n+1} \hbox{ is surjective.}$$
\end{proof}

\begin{cor} \label{homology2} Let $n \geq 1$ be a natural number,  $ A \to B \to C$ be a short exact sequence of groups  such that $A$ is of type $FP_n$ and $C $ is of type $FP_{n+1}$.  Then

a) if $H_0(C, H_n(A, \prod \mathbb{Z} B)) = 0$ for any  direct product $\prod \mathbb{Z} B$ then $B$ is of type $FP_{n+1}$;

b) if $C$ is of type $FP_{n+2}$ then $$B \hbox{ is of type }FP_{n+1} \hbox{ if and only if  }H_0(C, H_n(A, \prod \mathbb{Z} B)) = 0 $$ for any  direct product $\prod \mathbb{Z} B$.

\end{cor}

\begin{proof} Part a) follows directly from Proposition \ref{homology1}, since in this case the co-domain of $d_{n+1,0}^{n+1}$ is $0$, so $d_{n+1,0}^{n+1}$ is surjective.
 
To prove part b) note that if $C$ is of type $FP_{n+2}$ then $H_{n+1}(C, \prod \mathbb{Z} C) = \prod H_{n+1}(C, \mathbb{Z} C) = \prod  0 = 0$. Since $A$ is finitely generated $H_0(A, - )$ commutes with direct products, hence 
$$
 H_{n+1}(C, H_0(A, \displaystyle\prod\Z B)) \simeq 
 H_{n+1}(C,\displaystyle\prod  H_0(A, \Z B)) = H_{n+1}(C, \displaystyle\prod \Z C)) = 0.$$
 Thus the domain of the map $d_{n+1,0}^{n+1} $ is $0$, so
$d^{n+1}_{n+1, 0}$ is the zero map and $d_{n+1,0}^{n+1} $  is surjective if and only if  $H_0(C, H_n(A, \prod \mathbb{Z} B)) = 0 $. Finally by Proposition \ref{homology1} $B$ is $FP_{n+1}$ if and only if $d_{n+1,0}^{n+1} $  is surjective. 
\end{proof}

Below we restate and prove Theorem A.

\begin{theorem}  \label{teo-kuckuck.general}
Let $n \geq 1$ be a natural number, $A \hookrightarrow B \twoheadrightarrow C$ a short exact sequence of groups with  $A$ of type $FP_n$ and $C$ of type $FP_{n+1}$. Assume there is another short exact sequence of groups $A \hookrightarrow B_0 \twoheadrightarrow C_0$ with $B_0$ of type $FP_{n+1}$ and that there is a group homomorphism $\theta: B_0 \rightarrow B$ such that $\theta|_A = id_A$,  i.e. there is a commutative diagram of homomorphisms of groups $$\xymatrix{A \ \ar@{^{(}->}[r] \ar[d]_{id_A} & B_0 \ar@{->>}[r]^{\pi_0} \ar[d]_{\theta} & C_0 \ar@{.>}[d]^{\nu} \\ A \ \ar@{^{(}->}[r] & B \ar@{->>}[r]^{\pi} & C}$$Then $B$ is of type $FP_{n+1}$.
\end{theorem}

\begin{proof} We break the proof in several steps.

{\bf Step 1.} Consider the LHS spectral sequence $$E^2_{p,q} = H_p(C_0, H_q(A, \Z B)) \underset{p}{\Rightarrow} H_{p+q}(B_0, \Z B)$$
i.e. this is the standard Lyndon-Hoschild-Serre spectral sequence applied for the short exact sequence $A \to B_0 \to C_0$ and the $\Z B_0$-module $\Z B$, where we view 
$\Z B$ as a $\Z B_0$-module via $\theta$.
Note that $H_q(A, \Z B) = \textbf{0}$ for $q \geq 1$  since $\Z B$ is a free  $\Z A$-module and
\begin{equation} \label{eq-9}
H_0(A, \Z B) \cong \Z \otimes_{\Z A}\Z B \cong \Z (B/A) \cong \Z C.
\end{equation}
It follows that 
\begin{equation} \label{eq-1}
E_{p,q}^2 = \left\{
\begin{array}{ccccccc}
\textbf{0}& \mbox{if} & q \geq 1\\
H_p(C_0, \Z C) & \mbox{if} & q = 0,  
\end{array}
\right.
\end{equation} 
hence the spectral sequence collapses and $E_{n,0}^{\infty}  = E_{n,0}^2 = H_n(C_0, \Z C)$ for every $n \geq 0$.
By the convergence of the LHS spectral sequence there is a filtration
$$
\textbf{0} = \Phi^{-1}H_n \subseteq \Phi^{0}H_n \subseteq \ldots \subseteq \Phi^{n-1}H_n \subseteq \Phi^{n}H_n = H_n
$$
such that
$$
E^{\infty}_{p,q} \cong \Phi^pH_n/\Phi^{p-1}H_n \textrm{ for }   p+q=n,
$$
where $H_n$ denotes $H_n(B_0, \Z B) $.
Thus 
\begin{equation} \label{eq-filtr.general001}
\textbf{0} = \Phi^{-1}H_n = \Phi^{0}H_n = \ldots = \Phi^{n-1}H_n \subseteq \Phi^{n}H_n = H_n
\end{equation}
and so
the homomorphism \begin{equation} \label{eq-6}\varphi: H_n =  H_n(B_0, \Z B) \to E_{n,0}^{\infty} \simeq H_n(C_0, \Z C) \hbox{  is an isomorphism for } n \geq 0,\end{equation}  where $\varphi$ is induced by $\pi_0: B_0 \twoheadrightarrow C_0$ and by the homomorphism $\pi_{\#}: \Z B \twoheadrightarrow \Z C$ that itself is induced by $\pi$.

{\bf Step 2.} Let $I$ be  an  index set. Consider the LHS spectral sequence $$\mathcal{E}^2_{p,q} = H_p(C_0, H_q(A, \displaystyle\prod_{\alpha \in I}(\Z B)_{\alpha})) \underset{p}{\Rightarrow} H_{p+q}(B_0, \displaystyle\prod_{\alpha \in I}(\Z B)_{\alpha}),$$  associated to the short exact sequence of groups $A \to B_0 \to C_0$, where $(\Z B)_{\alpha} = \Z B$ for $\alpha \in I$ and $\Z B$ is $\Z B_0$-module via $\theta$.
Since $A$ is of type $FP_n$ by Lemma \ref{biericriterion}, it follows that 
\begin{equation} \label{newn} H_q(A, \displaystyle\prod_{\alpha \in I}(\Z B)_{\alpha}) = \displaystyle\prod_{\alpha \in I}H_q(A,(\Z B)_{\alpha})\hbox{ for }0 \leq q \leq n-1. \end{equation} Furthermore $H_q(A,(\Z B)_{\alpha}) = H_q(A,\Z B) = \textbf{0}$ for $q \geq 1$,  since $\Z B$ is a free $\Z A$-module. Then
\begin{equation} \label{eq-7}
\mathcal{E}^2_{p,q} = \textbf{0} \textrm{ if } 1 \leq q \leq n-1.
\end{equation}
Then since  $\mathcal{E}^{\infty}_{p,q}$  is a subquotient of  $\mathcal{E}^2_{p,q}$ we obtain that
\begin{equation} \label{eq-8}
\mathcal{E}^{\infty}_{p,q} = \textbf{0} \textrm{ if } 1 \leq q \leq n-1.
\end{equation} 
Observe that $C_0$ is of type  $FP_{n+1}$ by Proposition \ref{shortexact}, hence by Lemma \ref{biericriterion} the functor $H_n(C_0, - )$ commutes with direct products. Since $A$ is finitely generated the functor $H_0(A, - )$ commutes with direct products. This together with   (\ref{eq-9})  implies the isomorphisms $$\mathcal{E}^2_{n,0} =  H_n(C_0, H_0(A, \displaystyle\prod_{\alpha \in I}(\Z B)_{\alpha})) \cong  H_n(C_0, \displaystyle\prod_{\alpha \in I} H_0(A, (\Z B)_{\alpha})) \cong $$ \begin{equation} \label{kumon333} \displaystyle\prod_{\alpha \in I}H_n(C_0,  H_0(A, (\Z B)_{\alpha}))  \cong \displaystyle\prod_{\alpha \in I}H_n(C_0, (\Z C)_{\alpha}),\end{equation} where $(\Z C)_{\alpha} = \Z C$ and  $(\Z B)_{\alpha} = \Z B$, for every $\alpha \in I$.

{\bf Step 3.} Consider the differentials   $$\xymatrix{\mathcal{E}^i_{n+i,1-i} \ar[rr]^{\delta^i_{n+i,1-i}} & & \mathcal{E}^{i}_{n,0} \ar[r]^-{\delta^i_{n,0}} & \mathcal{E}^i_{n-i,i-1}}.$$ Then $\mathcal{E}^i_{n+i,1-i} = \textbf{0}$, since $1-i < 0$, and furthermore  $\mathcal{E}^i_{n-i,i-1} = \textbf{0}$ for $2 \leq i \leq n$ by (\ref{eq-7}). Note that $\mathcal{E}^i_{n-i,i-1} = \textbf{0}$ for $i > n$,  since in this case $n-i<0$. Thus   $$\mathcal{E}^{i+1}_{n,0} = \frac{ker(\delta^i_{n,0})}{im(\delta^i_{n+i,1-i})} = \mathcal{E}^i_{n,0} \textrm{ for every } i \geq 2.$$
This together with  (\ref{kumon333})  implies
\begin{equation} \label{eq-10}
\mathcal{E}^{\infty}_{n,0} = \mathcal{E}^2_{n,0} \simeq H_n(C_0, \displaystyle\prod_{\alpha \in I}(\Z C)_{\alpha}).
\end{equation}
By the convergence of the spectral sequence there is a filtration
\begin{equation} \label{eq-11}
\textbf{0} = \Lambda^{-1}H_n \subseteq \Lambda^{0}H_n \subseteq \ldots \subseteq \Lambda^{n-1}H_n \subseteq \Lambda^{n}H_n = H_n,
\end{equation}
such that
\begin{equation} \label{eq-12}
\mathcal{E}^{\infty}_{p,q} \cong \Lambda^pH_n/\Lambda^{p-1}H_n \textrm{ for }p+q = n,
\end{equation}
where to simplify the notation we denote $H_n(B_0,\displaystyle\prod_{\alpha \in I}(\Z B)_{\alpha})$ by $H_n$.  
By  (\ref{eq-8}) and (\ref{eq-12}) we have  $$\textbf{0} = \Lambda^{-1} H_n \subseteq  \Lambda^0H_n = \ldots = \Lambda^{n-1}H_n \subseteq \Lambda^nH_n = H_n = H_n(B_0, \displaystyle\prod_{\alpha \in I}(\Z B)_{\alpha}).$$ Thus, on one hand $$\mathcal{E}^{\infty}_{0,n} \cong \Lambda^0H_n/\Lambda^{-1}H_n  \cong \Lambda^0H_n = \Lambda^{n-1}H_n.$$ And on other hand $$\mathcal{E}^{\infty}_{n,0} \cong \Lambda^nH_n/\Lambda^{n-1}H_n = H_n/\Lambda^{n-1}H_n \cong H_n/\mathcal{E}^{\infty}_{0,n}.$$ Then we have a short exact sequence of groups 
\begin{equation} \label{ses123} \mathcal{E}^{\infty}_{0,n} \hookrightarrow H_n \overset{\widehat{\theta}}{\twoheadrightarrow} \mathcal{E}^{\infty}_{n,0},\end{equation} where the epimorphism $\widehat{\theta}$ is induced by the epimorphism $\pi_0: B_0 \twoheadrightarrow C_0$ e $\pi_{\#}: \Z B \twoheadrightarrow \Z C$, where $\pi_{\#}$ is a ring  epimorphism induced by the epimorphism  of groups $\pi$. 

{\bf Step 4.} We claim that $$\widehat{\theta} : H_n(B_0, \prod_{\alpha \in I} (\Z B)_{\alpha}) \to H_n(C_0, \prod_{\alpha \in I} (\Z C)_{\alpha}) $$ is an isomorphism. In fact, by (\ref{eq-6}), there is a group  isomorphism $\varphi$ that induces an isomorphism $$\Pi \varphi: \displaystyle\prod_{\alpha \in I}H_n(B_0, (\Z B)_{\alpha}) \overset{\sim}{\to} \displaystyle\prod_{\alpha \in I}H_n(C_0, (\Z C)_{\alpha}) .$$ 
Since $B_0$ and $C_0$ are $FP_{n+1}$, we have that both functors $H_n(B_0,  - )$ and $H_n(C_0, - )$ commute with direct products. Then, $\Pi \varphi$  induces the homomorphism of groups $\widehat{\theta}$.
Then by (\ref{ses123})
\begin{equation} \label{eq-13}
\mathcal{E}^{\infty}_{0,n} = Ker (\widehat{\theta}) = \textbf{0}.
\end{equation}

{\bf Step 5.} Consider the differentials for  $i \geq 2$ $$\xymatrix{\mathcal{E}^i_{i,n+1-i} \ar[rr]^{\delta^i_{i,n+1-i}} & & \mathcal{E}^{i}_{0,n} \ar[r]^-{\delta^i_{0,n}} & \mathcal{E}^i_{-i,n+i-1}}.$$ Then by  (\ref{eq-7}),  $\mathcal{E}^i_{i,n+1-i} = \textbf{0}$ for $2 \leq i \leq n$   and  $\mathcal{E}^i_{i,n+1-i} = \textbf{0}$ for $i \geq n+2$ since $n+1-i<0$. Furthermore $\mathcal{E}^i_{-i,n+i-1} = {0}$, since $-i < 0$ . Thus we have  $$\mathcal{E}^{i+1}_{0,n} = \frac{ker(\delta^i_{0,n})}{im(\delta^i_{i,n+1-i})} \cong \mathcal{E}^i_{0,n} \textrm{ for every } i \geq 2 \ \hbox{ and } \ i \neq n+1.$$
Then
\begin{equation} \label{eq-14}
\mathcal{E}^{2}_{0,n} \cong \ldots \cong \mathcal{E}^{n+1}_{0,n} \ \ \hbox{ and } \ \ \mathcal{E}^{n+2}_{0,n} \cong \ldots \cong \mathcal{E}^{\infty}_{0,n}.
\end{equation}
By (\ref{eq-13})  and (\ref{eq-14}) we have that 
\begin{equation} \label{eq-15}
\mathcal{E}^{n+2}_{0,n} \cong \mathcal{E}^{\infty}_{0,n} = \textbf{0}.
\end{equation}

{\bf Step 6.} 
Consider the differentials  $$\xymatrix{\mathcal{E}^{n+1}_{n+1,0} \ar[rr]^{\delta^{n+1}_{n+1,0}} & & \mathcal{E}^{n+1}_{0,n} \ar[r]^-{\delta^{n+1}_{0,n}} & \mathcal{E}^{n+1}_{-n-1,2n}}.$$ Note that $\mathcal{E}^{n+1}_{-n-1,2n} = \textbf{0}$, since $-n-1 < 0$, then $ker(\delta^{n+1}_{0,n}) = \mathcal{E}^{n+1}_{0,n}.$ This together with (\ref{eq-15}) implies $$\textbf{0} = \mathcal{E}^{n+2}_{0,n} := \frac{ker(\delta^{n+1}_{0,n})}{im(\delta^{n+1}_{n+1,0})} = \frac{\mathcal{E}^{n+1}_{0,n}}{im(\delta^{n+1}_{n+1,0})}.$$ Then $im(\delta^{n+1}_{n+1,0}) = \mathcal{E}^{n+1}_{0,n}$ and  we conclude that
\begin{equation} \label{eq-16}
\delta^{n+1}_{n+1,0}: \mathcal{E}^{n+1}_{n+1,0} \longrightarrow \mathcal{E}^{n+1}_{0,n} \textrm{  is surjective }. 
\end{equation}
As in the proof of Theorem \ref{teo-1.general} $\delta^{n+1}_{n+1,0}$ has the following domain and co-domain: $$\xymatrix{\delta^{n+1}_{n+1,0}: \mathcal{E}^{2}_{n+1,0} \ar@{->>}[rr] & & \mathcal{E}^{2}_{0,n}}.$$

{\bf Step 7.} By the naturality of the LHS spectral sequence, we have the following commutative diagram of group homomorphisms : 
\begin{equation} \label{eq-18}
\xymatrix{H_{n+1}(C_0,H_0(A, \displaystyle\prod_{\alpha \in I}(\Z B)_{\alpha})) \ar@{->>}[rr]^{\delta^{n+1}_{n+1,0}} \ar[d]_{\mu_1} & & H_0(C_0, H_n(A, \displaystyle\prod_{\alpha \in I}(\Z B)_{\alpha})) \ar[d]^{\mu_2} \\  H_{n+1}(C,H_0(A, \displaystyle\prod_{\alpha \in I}(\Z B)_{\alpha})) \ar[rr]_{\psi^{n+1}_{n+1,0}} & & H_0(C, H_n(A, \displaystyle\prod_{\alpha \in I}(\Z B)_{\alpha}))}
\end{equation}
where $\mu_1$ and $\mu_2$ are induced by $\nu : C_0 \to C$ and  $\psi^{n+1}_{n+1,0}$ is the differencial of the LHS spectral sequence  $$H_p(C, H_q(A, \displaystyle\prod_{\alpha \in I}(\Z B)_{\alpha})) \underset{p}{\Rightarrow} H_{p+q}(B, \displaystyle\prod_{\alpha \in I}(\Z B)_{\alpha})$$ associated to the short exact sequence $A \to B \to C$. Set $V =  H_n(A, \displaystyle\prod_{\alpha \in I}(\Z B)_{\alpha})$. 
Recall that $\nu : C_0 \to C$ is induced by $\theta$ and $V$ is a left  $\Z C_0$-module  via  the homomorphism $\nu$. Thus we have $$H_0(C_0,V) \cong \frac{V}{Aug(\Z C_0)V} = \frac{V}{Aug(\Z(im(\nu)))V} \ \ \ \hbox{ and } \ \ H_0(C,V) \cong \frac{V}{Aug(\Z C)V},$$
where $Aug$ denotes the augmentation ideal of the appropriate group algebra. Then $$\mu_2: \frac{V}{Aug(\Z(im(\nu)))V} \longrightarrow \frac{V}{Aug(\Z C)V}$$ is the enlargment homomorphism, hence is surjective.  By the commutative diagram (\ref{eq-18}), since $\delta^{n+1}_{n+1,0}$  and $\mu_2$ are both surjective, we deduce that $\psi^{n+1}_{n+1,0}$ is surjective too and hence by Proposition \ref{homology1} we have that $B$  is of type $FP_{n+1}$.
\end{proof}

Note that several results in \cite{Benno} are deduced as corollaries of the technical  \cite[Prop.~ 4.3]{Benno}. As we proved the homological version  Theorem \ref{teo-kuckuck.general} of \cite[Prop.~ 4.3]{Benno} we will deduce in the following propositions  that the homological versions of several results of \cite{Benno} hold too. The proofs of the following results will use significantly Theorem \ref{teo-kuckuck.general} plus ideas from \cite{Benno}. 

The following proposition implies Theorem B. As the statement  of Proposition \ref{split} shows, in the case when the second short exact sequence splits, there is no need to assume that $Q$ is of type $FP_{n+2}$ as in the Homological $n$-$(n+1)$-$(n+2)$ Conjecture. 

\begin{prop} \label{split}  Let $n \geq 1$ be a natural number, $N_1 \to G_1 \to Q$ and $N_2 \to G_2 \to Q$ be short exact sequences with $G_1$ and $G_2$ groups of type $FP_{n+1}$ such that $N_1$  is of type $FP_{n}$ and the second sequence splits. Then the fibre $P$, associated to the above short exact sequences, is of homological type $FP_{n+1}$. In particular the Homological $n$-$(n+1)$-$(n+2)$ Conjecture holds if the second sequence splits.
\end{prop}

\begin{proof} As in the proof of \cite[Cor,~4.6]{Benno} there is a homomorphism $\phi : G_1 \to P$ whose restriction to $N_1$ is the identity map.  Then  we obtain the following exact diagram $$\xymatrix{N_1 \ \ar@{^{(}->}[r] \ar[d]_{id_{N_1}} & \ G_1 \ar@{->>}^{\pi_1}[r] \ar[d]^{\phi} & Q \ar[d]^{\sigma_2} \\ N_1 \ \ar@{^{(}->}[r] & \ P \ar@{->>}^{p_2}[r] & G_2}$$
Finally by Teorema \ref{teo-kuckuck.general}, $P$  is of type $FP_{n+1}$.
\end{proof}

The following corollary proves Theorem C. The second part of Corollary \ref
{Q-fin-pres} will be used in the proof of Theorem \ref{homFinitePres}.

\begin{cor}  \label{Q-fin-pres}
If the  Homological $n$-$(n+1)$-$(n+2)$ Conjecture holds whenever $G_2$ is a finitely generated free group then it holds in general. If 
the  Homological $n$-$(n+1)$-$(n+2)$ Conjecture holds whenever $G_2$ is a finitely generated free group and $Q$ is finitely presented then it holds if $Q$ is finitely presented without restrictions on $G_2$.
\end{cor}

\begin{proof} We prove the first statement, the proof of the second statement is the same.
Let $N_1 \hookrightarrow G_1 \stackrel{\pi_1}{\twoheadrightarrow} Q$ and  $N_2 \hookrightarrow G_2 \stackrel{\pi_2}{\twoheadrightarrow} Q$ be short exact sequences of groups such that  $N_1$  is of type $FP_n$, $G_1$ and $G_2$ are of type $FP_{n+1}$  and $Q$ is of type $FP_{n+2}$. Let $P$  be the fibre product  $$P = \{(g_1, g_2) \in G_1 \times G_2 : \pi_1(g_1) = \pi_2(g_2)\}$$ and let  $p: F \twoheadrightarrow G_2$  be an epimorphism, where  $F$ is a finitely generated free group. Consider the short exact sequences of groups  $N_1 \hookrightarrow G_1 \stackrel{\pi_1}{\twoheadrightarrow} Q$ and $ker(\pi_2 \circ p) \hookrightarrow F \stackrel{\pi_2 \circ p}{\twoheadrightarrow} Q$,  and denote by $P'$ the fibre product  $$P' = \{(g_1,f) \in G_1 \times F : \pi_1(g_1) = \pi_2 \circ p(f)\}.$$ By assumption, the homological $n$-$(n+1)$-$(n+2)$ Conjecture holds whenever the middle group in the second exact sequence is finitely generated and free, so $P'$  is of type $FP_{n+1}$. Consider the following commutative diagram, whose rows are short exact sequences $$\xymatrix{N_1 \times \textbf{1} \ \ar@{^{(}->}[r] \ar[d]_{id_{N_1 \times \textbf{1}}} & \ P' \ar@{->>}^{p'_2}[r] \ar[d]^{id_{G_1} \times p} & F \ar[d]^{p} \\ N_1 \times \textbf{1} \ \ar@{^{(}->}[r] & \ P \ar@{->>}^{p_2}[r] & G_2}$$ where $p'_2: P' \twoheadrightarrow F$ is given by $p'_2(g_1,f) = f$, $p_2:P \twoheadrightarrow G_2$  is defined by  $p_2(g_1, g_2) = g_2$  and $id_{G_1} \times p: P' \rightarrow P$ is 
restriction of $id_{G_1} \times p: G_1 \times F \to G_1 \times G_2$. Then by Theorem \ref{teo-kuckuck.general}  $P$  is of type $FP_{n+1}$. 
\end{proof}

The following result is Theorem D from the introduction.

\begin{theorem}
Let $n \geq 1$, $N_1 \hookrightarrow G_1 \stackrel{\pi_1}{\twoheadrightarrow} Q $ and $N_2 \hookrightarrow G_2 \stackrel{\pi_2}{\twoheadrightarrow} Q$ be short exact sequences of groups, where $G_1, G_2$ are of type $FP_{n+1}$, $Q$ is virtually abelian, $N_1$ is of type $FP_k$  and $N_2$ is of type $FP_l$ for some $k,l \geq 0$ with $k + l \geq n$. Then the fibre product $P$ of $\pi_1$  and $\pi_2$ is of type $FP_{n+1}$.
\end{theorem}

\begin{proof} 

We assume first that $Q$  is abelian. Then $P \lhd (G_1 \times G_2)$  and  $(G_1 \times G_2)/P$ is abelian. By Theorem \ref{teo-subesf.grande.contida.em.sigma}, to complete the proof  we have to show  that for every  character $\chi: G_1 \times G_2 \rightarrow \R$ with $\chi(P) = 0$, we have $[\chi] \in \Sigma^{n+1}(G_1 \times G_2, \Z)$. Let $\chi$ be such a character. 
Observe that $G_1 \times G_2 = (G_1 \times \textbf{1})P$  and $G_1 \times G_2 = P(\textbf{1} \times G_2).$
Thus $\chi|_{(G_1 \times \textbf{1})} \neq 0$, otherwise $\chi: (G_1 \times \textbf{1})P \rightarrow \R$ would be the zero character. Similarly,  $\chi|_{(\textbf{1} \times G_2)} \neq 0$.

Since $Q$ is abelian, $G_1/N_1, G_2/N_2$ are abelian. Note that $N_1 \cup N_2 \subseteq P$ and $\chi(P) = 0$  imply that $\chi|_{(G_1 \times \textbf{1})}(N_1) = \textbf{0}$  and $\chi|_{(\textbf{1} \times G_2)}(N_2) = \textbf{0}$. Let $k', l' \geq 0$ be such that $k' \leq k$, $l' \leq l$  and $k'+l' = n$. Since $N_1$ is of type $FP_k$  and $N_2$  is of type $FP_l$, then $N_1$ is of type $FP_{k'}$  and $N_2$ is of type $FP_{l'}$. Thus, by Theorem \ref{teo-subesf.grande.contida.em.sigma}, $[\chi|_{(G_1 \times \textbf{1})}] \in \Sigma^{k'}(G_1, \Z)$  and $[\chi|_{(\textbf{1} \times G_2)}] \in \Sigma^{l'}(G_2, \Z)$. By Theorem \ref{teo-desig.meinert} $$[\chi] \in \Sigma^{k'+l'+1}(G_1 \times G_2, \Z) = \Sigma^{n+1}(G_1 \times G_2, \Z).$$

Now we consider the general case i.e.  there is a normal abelian subgroup $A$ of finite index in $Q$. 
Consider the short exact sequence of groups $N_1 \to \pi_1^{-1}(A) \stackrel{\pi_1}{\twoheadrightarrow} A$ and $N_2 \to \pi_2^{-1}(A) \stackrel{\pi_2}{\twoheadrightarrow} A$. 
Since $[G_1:\pi_1^{-1}(A)] = [Q:A] < \infty$  and $G_1$  is of type $FP_{n+1}$, it follows that $\pi_1^{-1}(A)$ is of type $FP_{n+1}$. Similarly $\pi_2^{-1}(A)$  is of type $FP_{n+1}$. By the abelian case discussed above, the fibre product  $\tilde{P}$ of $\tilde{\pi}_1$  and $\tilde{\pi}_2$  is of type $FP_{n+1}$. 
Note that $\tilde{P} = P \cap (\pi_1^{-1}(A) \times \pi_2^{-1}(A))$, hence $[P:\tilde{P}] < \infty$. Since going up or down with a finite index does not change the homological type  we deduce that $P$  is of type $FP_{n+1}$.
\end{proof}

We finish the section with the proof of the first part of Theorem F.

\begin{theorem} If $n \geq 2$ and the Homological $(n-1)$-$n$-$(n+1)$ Conjecture holds whenever $Q$ is virtually nilpotent then the Homological Virtual Surjection Theorem holds.
\end{theorem}

\begin{proof} The proof is similar to the proof of \cite[Thm.~3.10]{Benno}, where instead of \cite[Prop.~4.3]{Benno} we apply Theorem \ref{teo-kuckuck.general} and we swap the numbers $n$ and $k$ in the proof of \cite[Claim, Thm.~3.10]{Benno}. We sketch the proof. Let $P \subseteq G_1 \times \ldots \times G_k$ be as in the statement of the 
Homological Virtual Surjection Theorem. Then by substituting each $G_i$ with a subgroup of finite index if necessary, we can assume that $P \subseteq G_1 \times \ldots \times G_k$ is a subdirect product i.e. $p_i(P) = G_i$ for every $1 \leq i \leq k$. By \cite[Prop.~3.2]{B-H-M-S.1} or \cite[Lemma~3.2]{Benno} we obtain  that $G_i/ (P \cap G_i)$ is virtually nilpotent for every $i$. 

Let $T= p_{1,\ldots, k-1} (P)$, where $p_{1,\ldots, k-1} : G_1 \times \ldots \times G_k \to G_1 \times \ldots \times G_{k-1}$ is the canonical projection  and $N_{1, \ldots, k-1} = P \cap T$.
As in the proof of \cite[Thm.~3.10]{Benno} the fact that $P$  virtually surjects on $n$ factors implies that $N_{1, \ldots, k-1} \subseteq G_1 \times \ldots \times G_{k-1} $ virtually surjects on $n-1$ factors. If the Homological Virtual Surjection Theorem  holds for smaller values of $k$, i.e. we use induction on $k$, then  $N_{1,\ldots, k-1}$ is of type $FP_{n-1}$. Furthermore $T \subseteq G_1 \times \ldots \times G_{k-1}$ virtually surjects on $n$-tupples if $n \leq k-1$, so by induction on $k$  the group $T$ is of type $FP_n$. If $n \geq k$ we get that $n = k$ and in this case the Homological Virtual Surjection Theorem obviously holds.

As in the proof of \cite[Thm.~3.10]{Benno} and using \cite[Lemma~2.3]{Benno} for the subdirect product $P \subseteq T \times G_k$, we deduce that  $P$ is the fibre product associated to the short exact sequences $N_{1,\ldots, k-1} \to T \to Q$ and $N_k \to G_k \to Q$, where   $N_k = P \cap G_k$. Thus $Q \simeq G_k / N_k $ is virtually nilpotent.  
Then we can apply   the Homological $(n-1)$-$n$-$(n+1)$ Conjecture since $Q$  is virtually  nilpotent and obtain that $P$ is of type $FP_{n}$. 
\end{proof}

\section{The Homological $1$-$2$-$3$ Conjecture for finitely presented $Q$}

In this section we prove Theorem E.

\begin{theorem} \label{homFinitePres} The Homological $1$-$2$-$3$ Conjecture holds if in addition $Q$ is finitely presented.
\end{theorem} 

\begin{proof} Let
 $A \hookrightarrow G_1 \stackrel{\pi_1}{\twoheadrightarrow} Q$ and  $B \hookrightarrow G_2 \stackrel{\pi_2}{\twoheadrightarrow} Q$   be short exacts sequences of groups with $A$ finitely generated, $G_1$ and $G_2$ of type $FP_2$  and $Q$ is $FP_3$ and finitely presented. Denote by $P$ the associated  fibre product. We aim to show that $P$ is of type $FP_2$.
By Corollary \ref{Q-fin-pres} we can assume that $G_2$ is a free group $F$ with a finite free basis $X$. Let $$\langle X \mid \widetilde{R} \rangle, \hbox{ where } \widetilde{R}= \{ r_i(\underline{x}) \}_{i \in I_0},$$
be a finite presentation of $Q$.
Then there is  a presentation of the group $G_1$ 
$$
G_1 = \langle X \cup A_0 \mid   R_1 \cup R_2 \cup R_3 \rangle,
$$
where  $A_0 = \{ a_1, a_2, \ldots, a_k \}$ is a finite set of generators of $A$ and
 $$ R_1 = \{ r_i(\underline{x}) w_i (\underline{a})^{-1} \}_{i \in I_0}, R_2 = \{ a_j^x  v_{j,x}(\underline{a})^{-1} \}_{1 \leq j \leq k, x \in X \cup X^{-1}} \hbox{ and } R_3 = \{ z_j(\underline{a}) \} _{j \in J}$$
for some possibly infinite index set  $J$  and $ w_i (\underline{a}) $ ,$ v_{j,x}(\underline{a})$, $  z_j(\underline{a})$ are elements of the free group with a free basis $A_0$. 

Let $R$ be the normal subgroup of the free group $F(X\cup A_0)$ with a free basis $X \cup A_0$ generated as a normal subgroup by $R_1 \cup R_2 \cup R_3$. Then there is a short exact sequence of groups
$$R\mono F(X\cup A_0)\epi G_1.$$
 Since $G_1$ is $FP_2$ we obtain that $R/[R,R]$ is finitely generated as $\Z G_1$-module via conjugation. Hence
  there is a finite subset $J_0$ of $J$ such that 
\begin{equation} \label{FP_2condition}
R = R_0 [R,R],
\end{equation}
where $R_0$ is the normal closure of $R_1 \cup R_2 \cup R_{3,0}$ in the free group $F(X \cup A_0)$ and $R_{3,0} = \{ z_j(\underline{a}) \} _{j \in J_0} \subseteq R_3$.

Consider the group $\widetilde{G}_1 = F(X \cup A_0) / R_0$. Then there is a natural projection
$$
\pi : \widetilde{G}_1 = F(X \cup A_0)/ R_0 \to G_1 = F(X \cup A_0) / R
$$ with kernel $S = R / R_0$. By (\ref{FP_2condition}) we have
\begin{equation} \label{kernelFP2}
S = [S,S].
\end{equation}
Since $R_1 \cup R_2$ are relations in $\widetilde{G}_1$ we deduce that the subgroup $\widetilde{A}_1$ of $\widetilde{G}_1$ generated by the elements of $A_0$ is a normal subgroup of $\widetilde{G}_1$ such that $\widetilde{G}_1 / \widetilde{A}_1 \simeq Q$. Thus there is a short exact sequence of groups
\begin{equation}\label{tildeA1}
\widetilde{A}_1 \rightarrowtail \widetilde{G}_1 \buildrel{\widetilde{\pi}_1}\over\epi Q
\end{equation}
with both $\widetilde{G}_1$ and $Q$ finitely presented.
Denote by $\widetilde{P}_1$ the fibre of the short exact sequences (\ref{tildeA1}) and 
$B \rightarrowtail F \buildrel{\pi_2}\over\epi Q,$ thus
 $$\widetilde{P}_1=\{(h_1,h_2)\in \widetilde{G}_1\times F\mid \widetilde{\pi}_1(h_1)=\pi_2(h_2)\}.$$
 Since $Q$ is $FP_3$ and is finitely presented, it is $F_3$. Then  the 1-2-3 Theorem from \cite{B-H-M-S.1}  implies that $\widetilde{P}_1$ is finitely presented. 

Recall that $P$ is the fibre of the original short exact sequences $A \rightarrowtail G_1 \buildrel{\pi_1}\over\epi Q$
and $B \rightarrowtail F \buildrel{\pi_2}\over\epi Q$, i.e.
$$P=\{(h_1,h_2)\in G_1\times F\mid\pi_1(h_1)=\pi_2(h_2)\}.$$
The map $\pi \times id_F : \widetilde{G}_1 \times F \to G_1 \times F$ induces an epimorphism
$$\mu:\widetilde{P}_1\epi P$$
with kernel $ker(\mu)=S,$ where $S = ker (\pi)$.
 Then $\widetilde{P}_1=F(X \cup A_0) / S_1$ for some normal subgroup  $S_1$ of $F(X \cup A_0)$ and  $P = F(X \cup A_0) / S_2$ for some normal subgroup $S_2$  of $F(X \cup A_0)$ such that  $S_1 \subseteq S_2$ and $S_2/ S_1 = S = [S,S]$. Hence
$$
S_2 = S_1 [S_2, S_2],
$$
so $S_2/ [S_2, S_2]$ is an epimorphic image of the abelianization $S_1/ [S_1, S_1]$.
Since $\widetilde{P}_1$ is finitely presented, it is of type $FP_2$. Then we deduce that $S_1/ [S_1, S_1]$ is finitely generated as $\Z \widetilde{P}_1$-module via conjugation. Thus its epimorphic image $S_2/ [S_2, S_2]$ is finitely generated as $\Z P$-module via conjugation, hence $P$ is of type $FP_2$ as claimed.
\end{proof}

{\bf Acknowledgements} The first named author was partially supported by ''bolsa de produtividade em pesquisa'' 303350/2013-0 from CNPq, Brazil
 and grant 2016/05678-3 from FAPESP, Brazil. The second named author was supported by PhD grant from CAPES/CNPq, Brazil.

\end{document}